\documentclass{amsart}

\newcommand{\C}{\mathbb{C}}

\newcommand{\N}{\mathbb{N}}

\newcommand{\norm}[1]{\left\| #1 \right\|}

\newtheorem{theorem}{Theorem}[section]
\newtheorem{lemma}[theorem]{Lemma}

\theoremstyle{definition}
\newtheorem{definition}[theorem]{Definition}

\theoremstyle{theorem}

\theoremstyle{theorem}
\newtheorem{proposition}[theorem]{Proposition}

\theoremstyle{theorem}

\theoremstyle{theorem}

\theoremstyle{definition}

\theoremstyle{theorem}

\numberwithin{equation}{section}

\begin{document}
\title{A note on the computable categoricity of $l^p$ spaces}
\author{Timothy H. McNicholl}
\address{Department of Mathematics\\
Iowa State University\\
Ames, Iowa 50011}
\email{mcnichol@iastate.edu}

\begin{abstract}
Suppose that $p$ is a computable real and that $p \geq1$.  We show that in both the real and complex case, $\ell^p$ is computably categorical if and only if $p = 2$.  The proof uses Lamperti's characterization of the isometries of Lebesgue spaces of $\sigma$-finite measure spaces.  
\end{abstract}
\maketitle

\section{Introduction}\label{sec:intro}

When $p$ is a positive real number, let $\ell^p$ denote the space of all sequences of complex numbers $\{a_n\}_{n = 0}^\infty$ so that 
\[
\sum_{n = 0}^\infty |a_n|^p < \infty.
\]
$\ell^p$ is a vector space over $\C$ with the usual scalar multiplication and vector addition.  When $p \geq 1$ it is a Banach space under the norm defined by 
\[
\norm{\{a_n\}_n } = \left( \sum_{n = 0}^\infty |a_n|^p \right)^{1/p}.
\]  

Loosely speaking, a computable structure is \emph{computably categorical} if all of its computable copies are computably isomorphic.  In 1989, Pour-El and Richards showed that $\ell^1$ is not computably categorical \cite{Pour-El.Richards.1989}.  It follows from a recent result of A.G. Melnikov that $\ell^2$ is computably categorical \cite{Melnikov.2013}.  At the 2014 Conference on Computability and Complexity in Analysis, A.G. Melnikov asked ``For which computable reals $p \geq 1$ is $\ell^p$ computably categorical?"  The following theorem answers this question.

\begin{theorem}\label{thm:main.1}
Suppose $p$ is a computable real so that $p \geq 1$.  Then, $\ell^p$ is computably categorical if and only if $p = 2$.
\end{theorem}

We prove Theorem \ref{thm:main.1} by proving the following stronger result.

\begin{theorem}\label{thm:main.2}
Suppose $p$ is a computable real so that $p \geq 1$ and $p \neq 2$.  Suppose $C$ is a c.e. set.  Then, there is a computable copy of $\ell^p$, $\mathcal{B}$, so that $C$ computes a linear isometry of $\ell^p$ onto $\mathcal{B}$.  Furthermore, if an oracle $X$ computes a linear isometry of $\ell^p$ onto $\mathcal{B}$, then $X$ must also compute $C$.
\end{theorem}

These results also hold for $\ell^p$-spaces over the reals.  In a forthcoming paper it will be shown that $\ell^p$ is $\Delta_2^0$-categorical.

The paper is organized as follows.  Section \ref{sec:background} covers background and motivation.  Section \ref{sec:proof} presents the proof of Theorem \ref{thm:main.2}.  Concluding remarks are presented in Section \ref{sec:conclusion}.

\section{Background}\label{sec:background}

\subsection{Background from functional analysis}

Fix $p$ so that  $1 \leq p < \infty$.  A \emph{generating set} for $\ell^p$ is a subset of $\ell^p$ with the property that $\ell^p$ is the closure of its linear span.

Let $e_n$ be the vector in $\ell^p$ whose $(n+1)$st component is $1$ and whose other components are $0$.  Let $E = \{e_n\ :\ n \in \N\}$.  We call $E$ the \emph{standard generating set} for $\ell^p$.
 
Recall that an \emph{isometry} of $\ell^p$ is a norm-preserving map of $\ell^p$ into $\ell^p$.
We will use the following classification of the surjective linear isometries of $\ell^p$.

\begin{theorem}[Banach/Lamperti]\label{thm:classification}
Suppose $p$ is a real number so that $p \geq 1$ and $p \neq 2$.  Let $T$ be a linear map of $\ell^p$ into $\ell^p$.  Then, the following are equivalent.
\begin{enumerate}
	\item $T$ is a surjective isometry.
	
	\item There is a permutation of $\N$, $\phi$, and a sequence of unimodular points, $\{\lambda_n\}_n$, so that $T(e_n) = \lambda_n e_{\phi(n)}$ for all $n$.
	
	\item Each $T(e_n)$ is a unit vector and the supports of $T(e_n)$ and $T(e_m)$ are disjoint whenever $m \neq n$. 
\end{enumerate}
\end{theorem}

In his seminal text on linear operators, S. Banach stated Theorem \ref{thm:classification} for the case of $\ell^p$ spaces over the reals \cite{Banach.1987}.  He also stated a classification of the linear isometries of $L^p[0,1]$ in the real case.  Banach's proofs of these results were sketchy and did not easily generalize to the complex case.  In 1958, J. Lamperti rigorously proved a generalization of  Banach's claims to real and complex $L^p$-spaces of $\sigma$-finite measure spaces \cite{Lamperti.1958}.  Theorem \ref{thm:classification} follows from J.  Lamperti's work as it appears in Theorem 3.2.5 of \cite{Fleming.Jamison.2003}.   Note that Theorem \ref{thm:classification} fails when $p = 2$.  For, $\ell^2$ is a Hilbert space.  So, if $\{f_0, f_1, \ldots\}$ is any orthonormal basis for $\ell^2$, then there is a unique surjective linear isometry of $\ell^2$, $T$, so that $T(e_n) = f_n$ for all $n$.

\subsection{Background from computable analysis}

We assume the reader is familiar with the fundamental notions of computability theory as covered in \cite{Cooper.2004}.  

Suppose $z_0 \in \C$.  We say that $z_0$ is \emph{computable} if there is an algorithm that given any $k \in \N$ as input computes a rational point $q$ so that $|q - z_0| < 2^{-k}$.  This is equivalent to saying that the real and imaginary parts of $z_0$ have computable decimal expansions.

Our approach to computability on $\ell^p$ is equivalent to the format in \cite{Pour-El.Richards.1989} wherein a more expansive treatment may be found.

Fix a computable real $p$ so that $1 \leq p < \infty$.  Let $F = \{f_0, f_1, \ldots\}$ be a generating set for $\ell^p$.  We say that $F$ is an \emph{effective generating set} if there is an algorithm that given any rational points $\alpha_0, \ldots, \alpha_M$ and a nonnegative integer $k$ as input computes a rational number $q$ so that 
\[
q - 2^{-k} < \norm{\sum_{j = 0}^M \alpha_j f_j}< q + 2^{-k}.
\]
That is, the map 
\[
\alpha_0, \ldots, \alpha_M \mapsto \norm{\sum_{j = 0}^M \alpha_j f_j}
\]
is computable. Clearly the standard generating set is an effective generating set.  

Suppose $F = \{f_0, f_1, \ldots\}$ is an effective generating set for $\ell^p$.  We say that a vector $g \in \ell^p$ is \emph{computable with respect to $F$} if there is an algorithm that given any nonnegative integer $k$ as input computes rational points $\alpha_0, \ldots, \alpha_M$ so that 
\[
\norm{g - \sum_{j = 0}^M \alpha_j f_j} < 2^{-k}.
\]
Suppose $g_n \in \ell^p$ for all $n$.  We say that $\{g_n\}_n$, is \emph{computable with respect to $F$} if there is an algorithm that given any $k,n \in \N$ as input computes rational points $\alpha_0, \ldots, \alpha_M$ so that 
\[
\norm{g_n - \sum_{j = 0}^M \alpha_j f_j} < 2^{-k}.
\]

When $f \in \ell^p$ and $r > 0$, let $B(f ; r)$ denote the open ball with center $f$ and radius $r$.  When $\alpha_0, \ldots, \alpha_M$ are rational points and $r$ is a positive rational number, we call $B\left(\sum_{j = 0}^M \alpha_j f_j; r \right)$ a \emph{rational ball}.  

Suppose $F = \{f_0, f_1, \ldots\}$ and $G = \{g_0, g_1, \ldots\}$ are effective generating sets for $\ell^p$.  We say that a map $T: \ell^p \rightarrow \ell^p$ is \emph{computable with respect to $(F,G)$} if there is an algorithm $P$ that meets the following three criteria.
\begin{itemize}
	\item \bf Approximation:\rm\ Given a rational ball $B(\sum_{j = 0}^M \alpha_j f_j ; r)$ as input, $P$ either does not halt or produces a rational ball $B(\sum_{j = 0}^N \beta_j g_j; r')$.
	
	\item \bf Correctness:\rm\ If $B(\sum_{j = 0}^N \beta_j g_j ; r')$ is the output of $P$ on input\\ $B(\sum_{j = 0}^M \alpha_j f_j; r)$, then $T(f) \in B(\sum_{j = 0}^N \beta_j g_j; r')$ whenever \\
	$f \in B(\sum_{j = 0}^M \alpha_j f_j; r)$.
	
	\item \bf Convergence:\rm\ If $U$ is a neighborhood of $T(f)$, then $f$ belongs to a rational ball $B_1 = B(\sum_{j = 0}^M \alpha_j f_j; r)$ so that $P$ halts on $B_1$ and produces a rational ball that is included in $U$.
\end{itemize}

When we speak of an algorithm accepting a rational ball $B(\sum_{j = 0}^M \alpha_j f_j; r)$ as input, we of course mean that it accepts some representation of the ball such as a code of the sequence $(r, M, \alpha_0, \ldots, \alpha_M)$.  

All of these definitions have natural relativizations.  For example, if \\
$F = \{f_0, f_1, \ldots \}$ is an effective generating set, then we say that $X$ computes a vector $g \in \ell^p$ with respect to $F$ if there is a Turing reduction that given the oracle $X$ and an input $k$ computes rational points $\alpha_0, \ldots, \alpha_M$ so that $\norm{g - \sum_{j = 0}^M \alpha_j f_j} < 2^{-k}$.

\subsection{Background from computable categoricity} 

For the sake of motivation, we begin by considering the following simple example.  Let $\zeta$ be an incomputable unimodular point in the plane.  For each $n$, let $f_n = \zeta e_n$.  Let $F = \{f_0, f_1, \ldots\}$.  Thus, $F$ is an effective generating set.  However, the vector $\zeta e_0$ is computable with respect to $F$ even though it is not computable with respect to the standard generating set $E$.  In fact, the only vector that is computable with respect to $E$ and $F$ is the zero vector.  The moral of the story is that different effective generating sets may yield very different classes of computable vectors and sequences.  However, there is a surjective linear isometry of $\ell^p$ that is computable with respect to $(E,F)$; namely multiplication by $\zeta$.  Thus, $E$ and $F$ give the same computability theory on $\ell^p$ even though they yield very different classes of computable vectors.  This leads to the following definition.

\begin{definition}\label{def:lpcompcat}
Suppose $p$ is a computable real so that $p \geq 1$.  We say that $\ell^p$ is \emph{computably categorical} if for every effective generating set $F$ there is a surjective linear isometry of $\ell^p$ that is computable with respect to $(E,F)$.
\end{definition}

The definitions just given for $\ell^p$ can easily be adapted to any separable Banach space.  Suppose $G = \{g_0, g_1\ldots, \}$ is an effective generating set for a Banach space $\mathcal{B}$.  The pair $(\mathcal{B}, G)$ is called a \emph{computable Banach space}.  Suppose that $\mathcal{B}$ is linearly isometric to $\ell^p$, and let $T$ denote a linear isometric mapping of $\mathcal{B}$ onto $\ell^p$.  Let $f_n = T(g_n)$, and let $F = \{f_0, f_1, \ldots\}$.  Then, $F$ is an effective generating set for $\ell^p$, and $T$ is computable with respect to $(G, F)$.  Thus, Theorem \ref{thm:main.2} can be rephrased as follows.

\begin{theorem}\label{thm:main.2'}
Suppose $p$ is a computable real so that $p \geq 1$ and $p \neq 2$.  Suppose $C$ is a \emph{c.e.} set.  Then, there is an effective generating set for $\ell^p$, $F$, so that with respect to $(E,F)$, $C$ computes a surjective linear isometry of $\ell^p$ .  Furthermore, any oracle that computes a surjective linear isometry of $\ell^p$ with respect to $(E,F)$ must also compute $C$.
\end{theorem} 

A.G. Melnikov and K.M. Ng have investigated computable categoricity questions with regards to the space $C[0,1]$ of continuous functions on the unit interval with the supremum norm \cite{Melnikov.2013}, \cite{Melnikov.Ng.2014}.  The study of computable categoricity for countable structures goes back at least as far as the work of Goncharov \cite{Goncharov.1975}.  The text of Ash and Knight has a thorough discussion of the main results of this line of inquiry \cite{Ash.Knight.2000}.  The survey by Harizanov covers other directions in the countable computable structures program \cite{Harizanov.1998}.

\section{Proof of Theorems \ref{thm:main.1} and \ref{thm:main.2}}\label{sec:proof}

We begin by noting the following easy consequence of the definitions and Theorem \ref{thm:classification}.

\begin{proposition}\label{prop:comp.sli}
Suppose $p$ is a computable real so that $p \geq 1$ and so that $p \neq 2$.  Let $F$ be an effective generating set for $\ell^p$.  Then, the following are equivalent.
\begin{enumerate}
	\item There is a surjective linear isometry of $\ell^p$ that is computable with respect to $(E,F)$.  \label{prop:comp.sli::itm.1}
	
	\item There is a permutation of $\N$, $\phi$, and a sequence of unimodular points $\{\lambda_n\}_n$, so that $\{\lambda_n e_{\phi(n)}\}_n$ is computable with respect to $F$. \label{prop:comp.sli::itm.2}
	
	\item There is a sequence of unit vectors $\{g_n\}_n$ so that $\{g_n\}_n$ is computable with respect to $F$, $G = \{g_0, g_1, \ldots\}$ is a generating set for $\ell^p$, and so that the supports of $g_n$ and $g_m$ are disjoint whenever $n \neq m$.  \label{prop:comp.sli::itm.3}
\end{enumerate}
\end{proposition}

\begin{proof}
Parts (\ref{prop:comp.sli::itm.2}) and (\ref{prop:comp.sli::itm.3}) just restate each other.  It follows from Theorem \ref{thm:classification} that (\ref{prop:comp.sli::itm.1}) implies (\ref{prop:comp.sli::itm.2}).  

Suppose (\ref{prop:comp.sli::itm.3}) holds.  Let $T$ be the unique linear map of the span of $E$ onto the span of $G$ so that $T(e_n) = g_n$ for all $n$.  Since the supports of $g_0, g_1, \ldots$ are pairwise disjoint, and since each $g_n$ is a unit vector, $T$ is isometric.  It follows that there is a unique extension of $T$ to a unique linear isometry of $\ell^p$; denote this extension by $T$ as well.  
We claim that $T$ is computable with respect to $(E,F)$.   For, suppose a rational ball $B(\sum_{j=0}^M \alpha_j e_j; r)$ is given as input.  Since $\{g_n\}_n$ is computable with respect to $F$, it follows that we can compute a non-negative integer $N$ and rational points $\beta_0, \ldots, \beta_N$ so that $\norm{\sum_{j = 0}^M \alpha_j g_j - \sum_{j = 0}^N \beta_j f_j} < r$.  We then output $B(\sum_{j = 0}^N \beta_j g_j; 2r)$.  It follows that the Approximation, Correctness, and Convergence criteria are satisfied and so $T$ is computable with respect to $(E, F)$. \
\end{proof}

We now turn to the proof of Theorem \ref{thm:main.2'} which, as we have noted, implies Theorem \ref{thm:main.2}.  Our construction of $F$ is a modification of the construction used by Pour-El and Richards to show that $\ell^1$ is not computably categorical \cite{Pour-El.Richards.1989}.
Let $C$ be an incomputable c.e. set.  Without loss of generality, we assume $0 \not \in C$.  Let $\{c_n\}_{n \in \N}$ be an effective enumeration of $C$.  Set
\[
\gamma = \sum_{k \in C} 2^{-k}.
\]
Thus, $0 < \gamma < 1$, and $\gamma$ is an incomputable real.  Set:
\begin{eqnarray*}
f_0 & = & (1 - \gamma)^{1/p} e_0 + \sum_{n = 0}^\infty 2^{- c_n / p} e_{n + 1}\\
f_{n + 1} & = & e_{n + 1}\\
F & = & \{f_0, f_1, f_2, \ldots \}
\end{eqnarray*}
Since $1 - \gamma > 0$, we can use the standard branch of $\sqrt[p]{\ }$.  

We divide the rest of the proof into the following lemmas.

\begin{lemma}\label{lm:EGS}
$F$ is an effective generating set.
\end{lemma}

\begin{proof}
Since 
\[
(1 - \gamma)^{1/p} e_0 = f_0 - \sum_{n =1}^\infty 2^{-c_{n-1} / p} f_n
\]
the closed linear span of $F$ includes $E$.  Thus, $F$ is a generating set for $\ell^p$.  Note that $\norm{f_0} = 1$.

Suppose $\alpha_0, \ldots, \alpha_M$ are rational points.  When $1 \leq j \leq M$, set
\[
\mathcal{E}_j = |\alpha_0 2^{-c_{j-1} / p} + \alpha_j |^p - |\alpha_0|^p 2^{-c_{j-1}}.
\]
It follows that 
\begin{eqnarray*}
\norm{\alpha_0 f_0 + \ldots +\alpha_M f_M}^p & = & |\alpha_0|^p \norm{f_0}^p + \mathcal{E}_1 + \ldots + \mathcal{E}_M\\
& = & |\alpha_0|^p + \mathcal{E}_1 + \ldots + \mathcal{E}_M.
\end{eqnarray*}
Since $\mathcal{E}_1$, $\ldots$, $\mathcal{E}_M$ can be computed from $\alpha_0, \ldots, \alpha_M$, $\norm{\alpha_0 f_0 + \ldots + \alpha_M f_M}$ can be computed from $\alpha_0, \ldots, \alpha_M$.  Thus, $F$ is an effective generating set.
\end{proof}

\begin{lemma}\label{lm:X.computes.C.0}
Every oracle that with respect to $F$ computes a
scalar multiple of $e_0$ whose norm is $1$ must also compute $C$.
\end{lemma}

\begin{proof}
Suppose that with respect to $F$, $X$ computes a vector of the form $\lambda e_0$ where $|\lambda| = 1$.  It suffices to show that $X$ computes $(1 - \gamma)^{-1/p}$.  

Fix a rational number $q_0$ so that $(1 - \gamma)^{-1/p} \leq q_0$.  Let $k \in \N$ be given as input.  Compute $k'$ so that $2^{-k'} \leq q_0 2^{-k}$.  Since $X$ computes $\lambda e_0$ with respect to $F$, we can use oracle $X$ to compute rational points $\alpha_0, \ldots, \alpha_M$ so that 
\begin{equation}
\norm{\lambda e_0 - \sum_{j = 0}^M \alpha_j f_j} < 2^{-k'}.\label{inq:1}
\end{equation}
We claim that $|(1 - \gamma)^{-1/p} - |\alpha_0| | < 2^{-k}$.  For, it follows from (\ref{inq:1}) that $|\lambda - \alpha_0 (1 - \gamma)^{1/p}| < 2^{-k'}$.  Thus, $|1 - |\alpha_0| (1 - \gamma)^{1/p}| < 2^{-k'}$.  Hence, 
\[
|(1 - \gamma)^{-1/p} - |\alpha_0|| < 2^{-k'}(1 - \gamma)^{-1/p} \leq 2^{-k'}q_0  \leq 2^{-k}.
\]
Since $X$ computes $\alpha_0$ from $k$, $X$ computes $(1 - \gamma)^{-1/p}$.
\end{proof}

\begin{lemma}\label{lm:X.computes.C.1}
If $X$ computes a surjective linear isometry of $\ell^p$ with respect to $(E,F)$, then $X$ must also compute $C$. 
\end{lemma}

\begin{proof}
By Lemma \ref{lm:X.computes.C.0} and the relativization of Proposition \ref{prop:comp.sli}.
\end{proof}

\begin{lemma}\label{lm:C.computes.e_0}
With respect to $F$, $C$ computes $e_0$.
\end{lemma}

\begin{proof}
Fix an integer $M$ so that $(1 - \gamma)^{-1/p} < M$.  

Let $k \in \N$.  Using oracle $C$, we can compute an integer $N_1$ so that $N_1 \geq 3$ and 
\[
\norm{ \sum_{n = N_1}^\infty 2^{-c_{n - 1}/p} e_n } \leq \frac{2^{-(kp + 1)/p}}{2^{-(kp + 1)/p} + M}.
\]
We can use oracle $C$ to compute a rational number $q_1$ so that $|q_1 - (1 - \gamma)^{-1/p}| \leq 2^{-(kp + 1)/p}$.  Set
\[
g = q_1 \left[ f_0 - \sum_{n = 1}^{N_1 - 1} 2^{-c_{n-1}/p} f_n \right].
\]
It suffices to show that $\norm{e_0 - g} < 2^{-k}$.  Note that since $1 - \gamma < 1$,\\ $|q_1(1 - \gamma)^{1/p} - 1| \leq 2^{-(kp + 1)/p}$.  Note also that $|q_1| < M + 2^{-(kp +1)/p}$.   Thus, 
\begin{eqnarray*}
\norm{e_0 - g}^p & = & \norm{e_0 - q_1(1 - \gamma)^{1/p} e_0 - q_1 \sum_{n = N_1}^\infty 2^{-c_{n - 1/p}}e_n}^p\\
& \leq & |q_1 (1 - \gamma)^{1/p} - 1|^p + |q_1|^p \norm{\sum_{n = N_1}^\infty 2^{-c_{n-1}/p} e_n}^p\\
& < & 2^{-(kp + 1)} + 2^{-(kp + 1)} = 2^{-kp}
\end{eqnarray*}
Thus, $\norm{e_0 - g} < 2^{-k}$.  This completes the proof of the lemma.
\end{proof}

\begin{lemma}\label{lm:C.computes.identity}
With respect to $(E,F)$, $C$ computes a surjective linear isometry of $\ell^p$.
\end{lemma}

\begin{proof}
By Lemma \ref{lm:C.computes.e_0} and the relativization of Proposition \ref{prop:comp.sli}.
\end{proof}

\section{Concluding remarks}\label{sec:conclusion}

We note that all of the steps in the above proofs work just as well over the real field.  

Lamperti's result on the isometries of $L^p$ spaces hold when $0 < p < 1$.  For these values of $p$,  $\ell^p$ is a metric space under the metric
\[
d(\{a_n\}_n, \{b_n\}_n) = \sum_{n = 0}^\infty |a_n - b_n|^p.
\]
The steps in the above proofs can be adapted to these values of $p$ as well.  

In a forthcoming paper it will be shown that $\ell^p$ is $\Delta_2^0$-categorical.

\section*{Acknowledgement}

The author thanks the anonymous referees who made helpful comments.  The author's participation in CiE 2015 was funded by a Simons Foundation Collaboration Grant for Mathematicians.

\def\cprime{$'$}


\begin{thebibliography}{10}
\providecommand{\url}[1]{\texttt{#1}}
\providecommand{\urlprefix}{URL }

\bibitem{Ash.Knight.2000}
Ash, C.J., Knight, J.: Computable structures and the hyperarithmetical
  hierarchy, Studies in Logic and the Foundations of Mathematics, vol. 144.
  North-Holland Publishing Co., Amsterdam (2000)

\bibitem{Banach.1987}
Banach, S.: Theory of linear operations, North-Holland Mathematical Library,
  vol.~38. North-Holland Publishing Co., Amsterdam (1987), translated from the
  French by F. Jellett, With comments by A. Pe{\l}czy{\'n}ski and Cz. Bessaga

\bibitem{Cooper.2004}
Cooper, S.B.: Computability theory. Chapman \& Hall/CRC, Boca Raton, FL (2004)

\bibitem{Fleming.Jamison.2003}
Fleming, R.J., Jamison, J.E.: Isometries on {B}anach spaces: function spaces,
  Chapman \& Hall/CRC Monographs and Surveys in Pure and Applied Mathematics,
  vol. 129. Chapman \& Hall/CRC, Boca Raton, FL (2003)

\bibitem{Goncharov.1975}
Goncharov, S.: Autostability and computable families of constructivizations.
  Algebra and Logic  17,  392--408 (1978), english translation

\bibitem{Harizanov.1998}
Harizanov, V.S.: Pure computable model theory. In: Handbook of recursive
  mathematics, {V}ol.\ 1, Stud. Logic Found. Math., vol. 138, pp. 3--114.
  North-Holland, Amsterdam (1998)

\bibitem{Lamperti.1958}
Lamperti, J.: On the isometries of certain function-spaces. Pacific J. Math.
  8,  459--466 (1958)

\bibitem{Melnikov.2013}
Melnikov, A.G.: Computably isometric spaces. J. Symbolic Logic  78(4),
  1055--1085 (2013)

\bibitem{Melnikov.Ng.2014}
Melnikov, A.G., Ng, K.M.: Computable structures and operations on the space of
  continuous functions, available at
  https://dl.dropboxusercontent.com/u/4752353/Homepage/C[0,1]\underline{\
  }final.pdf

\bibitem{Pour-El.Richards.1989}
Pour-El, M.B., Richards, J.I.: Computability in analysis and physics.
  Perspectives in Mathematical Logic, Springer-Verlag, Berlin (1989)

\end{thebibliography}
\end{document}